\documentclass[12pt]{article}
\usepackage{amsmath}
\usepackage[usenames]{color}
\usepackage{mathrsfs}
\usepackage{amsfonts}
\usepackage{amssymb,amsmath}
\usepackage{CJK}
\usepackage{cite}
\usepackage{cases}
\usepackage{amsthm}

\def\cl{\centerline}
\def\la{\lambda}
\def\al{\alpha}
\def\be{\beta}
\def\vs{\vspace*}

\def\Z{\mathbb{Z}}
\def\C{\mathbb{C}}

\def\N{\mathbb{N}}

\numberwithin{equation}{section}
\newtheorem{theo*}{Theorem}
\newtheorem{theo}{Theorem}[section]
\newtheorem{defi}[theo]{Definition}

\newtheorem{prop}[theo]{Proposition}

\newtheorem{clai}{Claim}
\newtheorem{case}{Case}

\def\W{\mathcal{W}}
\def\V{\mathcal{V}}

\begin{document}
\begin{CJK*}{GBK}{song}

\begin{center}
{\bf\large Modules over the algebra $\V ir(a,b)$}
\footnote {$^*\,$Supported by NSF grant no. 11371278, 11431010, 11501417, 
the Fundamental Research Funds for the Central Universities of China, Innovation Program of Shanghai Municipal Education Commission and  Program for Young Excellent Talents in Tongji University.

\ $^\dag\,$Correspondence: Q.~ Chen (Email: qiufantj@126.com)
}
\end{center}

\cl{Jianzhi Han$^{\,*}$, Qiufan Chen$^{\,\dag}$, Yucai Su$^{\,*}$}

\cl{\small Department of Mathematics, Tongji University, Shanghai 200092, China}

\vs{8pt}

{\small
\parskip .005 truein
\baselineskip 3pt \lineskip 3pt

\noindent{{\bf Abstract:} For any two complex numbers $a$ and $b$, $\V ir(a,b)$  is a central extension of $\W(a,b)$ which is universal in the case $(a,b)\neq (0,1)$, where $\W(a,b)$ is the Lie algebra with basis $\{L_n,\,W_n\mid\,n\in\Z\}$ and relations $[L_m,L_n]=(n-m)L_{m+n}$, $[L_m,W_n]=(a+n+bm)W_{m+n}$, $[W_m,W_n]=0$. In this paper, we construct and classify a class of non-weight modules over the algebra $\V ir(a,b)$ which are free $U(\C L_0\oplus\C W_0)$-modules of rank $1$. It is proved that such modules  can only exist for $a=0$.  \vs{5pt}

}

}
\parskip .001 truein\baselineskip 6pt \lineskip 6pt
\section{Introduction}
 For $a,b\in\C$,  the Lie algebra $\W(a,b)$  was studied in  \cite{OC}, which has a $\C$-basis $\{L_n,\,W_n\mid n\in\Z\}$ subject to relations
\begin{eqnarray}
&&\label{rel1} [L_m,L_n]=(n-m)L_{m+n},\\
&&\label{rel2}  [L_m,W_n]=(a+n+bm)W_{m+n},\\
&&\label{rel3}  [W_m,W_n]=0, \quad \forall m,n\in\Z.
\end{eqnarray}
It is clear that $\W(a,b)\cong W\ltimes A_{a,b}$, where $W$ is the Witt algebra and $A_{a,b}$ \cite{KS} is a module of the intermediate series over $W$.
It was shown in \cite{GJP} that $\W(a,b)$ is perfect if and only if $(a,b)\neq (0,1)$, in which case the universal central extension $\mathcal{V}ir(a,b)$ of $\W(a,b)$ were also determined therein. The algebra $\V ir(a,b)$ is very meaningful in the sense that it generalizes many important algebras, such as  the so-called  Heisenberg-Virasoro algebra $\V ir(0,0)$ which plays an important role in the representation theory of toroidal Lie algebra,  the $W(2,2)\, $ algebra $(=\V ir(0,-1))$ whose representations were studied in \cite{ZD} in terms of vertex operator algebras, and the so-called spin-$l$  algebra $\W(0,l)$ for $l\in \frac{1}{2}\N$ $(\N$ = the set of all positive integers).

The full subcategory  $\mathfrak{M}$ of $U(\mathfrak{sl_{n+1}})$-modules consisting of objects whose restriction to $U(\mathfrak{h})$ are free of rank $1$ ($\mathfrak{h}$ is the standard Cartan subalgebra of $\mathfrak{sl_{n+1}}$), was investigated in \cite{N},  isomorphism classes of objects in $\mathfrak{M}$  were classified and their irreducibilities were  determined therein.   Such kind of modules for finite dimensional simple Lie algebras were studied later in \cite{N1}. And the idea  was  exploited and generalized to consider modules over infinite Lie algebras, such as the Witt algebras of all ranks \cite{TZ},  Heisenberg-Virasoro algebra and $W(2,2)$ algebra  \cite{CG},  Lie  algebras related to  Virasoro algebra \cite{CC}. In the present paper,  we shall consider a class of non-weight modules over the algebra $U\big(\V ir(a,b)\big)$ whose restriction to $U(\C L_0\oplus\C W_0)$ are free of rank 1 and classify all such kind of modules.

This paper is organized as follows. In Section 2, we construct a class of non-weight $\V ir(0,b)$-modules  over $\C[s,t]$, and study the irreducibilities and isomorphic relations of these modules. Section 3 is devoted to classifying all modules over the algebra $U\big(\W (a,b)\big)$ and then over $U\big(V ir (a,b)\big)$  whose restriction to $U(\C L_0\oplus\C W_0)$ are free of rank 1. As a coproduct, we show that such modules exist only for $\W(0,b)$ and $\V ir(0, b)$.

Throughout the paper, the symbols $\C ,\,\Z,\, \C^*,\,\Z^*$ represent for the sets of complex numbers, integers, nonzero complex numbers and nonzero integers,  respectively. $U(\mathfrak{g})$ is used to denote the universal enveloping algebra of a Lie algebra $\mathfrak{g}$.
\section{Preliminaries}
For $(a,b)\neq (0,1)$, the universal central extension $\V ir(a,b)$ of $\W(a,b)$ are  determined in \cite{GJP}, which can be divided into the following four cases.

\textbf{Case  (i).} $\V ir(0,0)$ is the Lie algebra with the $\C$-basis $\{L_n,\,W_n,\,C_1,\,C_2,\,C_3\mid n\in\Z\}$ and the Lie brackets given by
\begin{equation*}
\aligned
&[L_n,L_m]=(m-n)L_{m+n}+\delta_{m+n,0}\frac{n^3-n}{12}C_1,\\
&[L_n,W_m]=mW_{m+n}+\delta_{m+n,0}(n^2+n)C_2,\quad [W_n,W_m]=n\delta_{m+n,0}C_3,\\
&[C_i,V ir(0,0)]=0,\quad \forall m,n\in\Z,\,i=1,2,3.
\endaligned
\end{equation*}

\textbf{Case  (ii).} $\V ir(0,-1)$ is the Lie algebra with the $\C$-basis $\{L_n,\,W_n,\,C_1,\,C_2\mid n\in\Z\}$ and the Lie brackets given by
\begin{equation*}
\aligned
&[L_n,L_m]=(m-n)L_{m+n}+\delta_{m+n,0}\frac{n^3-n}{12}C_1,\\
&[L_n,W_m]=(m-n)W_{m+n}+\delta_{m+n,0}\frac{n^3-n}{12}C_2,\\
&[W_n,W_m]=[C_i,\V ir(0,-1)]=0,\quad {\rm where}\ m,n\in\Z,\,i=1,2.
\endaligned
\end{equation*}

\textbf{Case  (iii).} $\V ir(\frac{1}{2},0)$ is the Lie algebra with the $\C$-basis $\{L_n,\,W_n,\,C_1,\,C_3\mid n\in\Z\}$ and the Lie brackets given by
\begin{equation*}
\aligned
&[L_n,L_m]=(m-n)L_{m+n}+\delta_{m+n,0}\frac{n^3-n}{12}C_1,\\
&[L_n,W_m]=(m+\frac{1}{2})W_{m+n},\quad [W_n,W_m]=(2n+1)\delta_{m+n,-1}C_3\\
&[C_i,\V ir(\frac{1}{2},0)]=0,\quad {\rm where}\ m,n\in\Z,\,i=1,3.\\
\endaligned
\end{equation*}

\textbf{Case  (iv).} $\V ir(a,b)$ with $(a,b)\neq (0,0), (0,\pm 1), (\frac{1}{2},0)$ is the Lie algebra with the $\C$-basis $\{L_n,\,W_n,\,C_1\mid n\in\Z\}$ and the Lie brackets given by
\begin{equation*}
\aligned
&[L_n,L_m]=(m-n)L_{m+n}+\delta_{m+n,0}\frac{n^3-n}{12}C_1,\\
&[L_n,W_m]=(m+a+nb)W_{m+n},\\
&[W_n,W_m]=[C_1,\V ir(a,b)]=0,\quad \forall m,n\in\Z.
\endaligned
\end{equation*}
 While the Lie algebra $\V ir(0,1)$   has a $\C$-basis $\{L_n,\,W_n,\,C_1,\,C_2,\,C_4\mid n\in\Z\}$ subject to
\begin{equation*}
\aligned
&[L_n,L_m]=(m-n)L_{m+n}+\delta_{m+n,0}\frac{n^3-n}{12}C_1,\\
&[L_n,W_m]=(m+n)W_{m+n}+\delta_{m+n,0}(n C_2+C_4),\\
&[W_n,W_m]=[C_i,\V ir(0,1)]=0,\quad {\rm where}\  m,n\in\Z,\,i=1,2,4.
\endaligned
\end{equation*}


Fix any $\al\in\C$ and $b\in\C$. For any nonnegative integer $k$ and $n\in\Z$, define the following polynomials
\begin{equation}\label{ew}q_{n,k;\al}(t)=nt^k-\delta_{b,-1}n(n-1)\al\frac{t^k-\al^k}{t-\al}-\delta_{b,1}n\al\frac{t^k-\al^k}{t-\al}.\end{equation}
Set
\begin{equation}\label{noas}\mathcal{H}_{\al}=\big\{\big(h_n(t)\big)_{n\in\Z}\mid h_n(t)=\sum_{i=0}^{\infty}{h}^{(i)} q_{n,i;\al}(t)\in\C[t], {h}^{(i)}\in\C\big\}.\end{equation}

\begin{defi}\label{defi2.1}\rm Let $\C[s,t]$ be the polynomial algebra in variables $s$ and $t$ with coefficients in $\C$. For $\lambda\in\C^*,\,\al\in\C$ and $\mathbf{h}=\big(h_n(t)\big)_{n\in\Z}\in\mathcal{H}_{\al}$, define the action of $\V ir(0,b)$ on $\Phi(\lambda,\al,\mathbf {h}):=\C[s,t]$ as follows:
\begin{eqnarray*}
L_m\cdot f(s,t)&=&\lambda^m\big(s+h_m (t)\big)f(s-m,t)+bm\lambda^m\big(t-\delta_{b,-1}m\al\\ &&-\delta_{b,1}(1-\delta_{m,0})\al\big)\frac{\partial}{\partial t}\big(f(s-m,t)\big),\\
 W_m\cdot f(s,t)&=&
\lambda^m\big(t-\delta_{b,-1}m\al-\delta_{b,1}(1-\delta_{m,0})\al\big)f(s-m,t),\\
C_i\cdot f(s,t)&=&0,\quad {\rm where}\  f(s,t)\in\C[s,t],\, i=1,2,3,4.
\end{eqnarray*}   For any $\lambda\in\C^*$ and $\mathbf{r}=\big(r_n(t)\big)_{n\in\Z}\in\mathcal{H}_0$, another action of $\V ir(0,1)$  on $\Theta(\lambda,\mathbf {r}):=\C[s,t]$ is defined by
\begin{eqnarray*}
L_m\cdot f(s,t)&=&\lambda^m\big(s+h_m (t)\big)f(s-m,t)\\
W_m\cdot f(s,t)&=&\delta_{m,0}tf(s-m,t)\\
 C_i\cdot f(s,t)&=&0,\quad {\rm where}\ f(s,t)\in\C[s,t],\, i=1,2,4.
\end{eqnarray*}  
\end{defi}
\begin{prop}\label{Pro}
All $\Phi(\lambda,\al,\bf {h})$  are $\V ir(0,b)$-modules and  all $\Theta(\lambda,\bf {r})$ are $\V ir(0,1)$-modules under the actions given in Definition \ref{defi2.1}.
\end{prop}
\begin{proof}
We only tackle with the case  $b\neq\pm1$,   the arguments for the other two cases analogue being similar.  According to the above definition, we have
\begin{equation*}
\aligned
&L_m\cdot f(s,t)=\lambda^m\big(s+h_m (t)\big)f(s-m,t)+\lambda^mbmt\frac{\partial}{\partial t}\big(f(s-m,t)\big),\\
&W_m\cdot f(s,t)=\lambda^mtf(s-m,t),\endaligned
\end{equation*}where $m\in\Z$ and $h_m (t)=\sum_{i=0}^{\infty}{h}^{(i)} mt^i$ for ${h}^{(i)}\in\C$.
By a straightforward computation, we obtain the following
\begin{eqnarray}\label{A11111}
&&L_m\cdot L_n\cdot \big(f(s,t)\big)\nonumber\\ &=&
L_m\cdot \big(\la^n(s+h_n (t))f(s-n,t)+\la^nbnt\frac{\partial}{\partial t}\big(f(s-n,t)\big)\nonumber\\
&=&\la^{m+n}\Big\{(s+h_m (t))(s-m+h_n (t))f(s-m-n,t)+ bnt(s+h_m (t))\frac{\partial}{\partial t}\big(f(s-m-n,t)\big)\nonumber\\
&& \ \ \ \ \ \ \ \ \ +\,bmtf(s-m-n,t)h^\prime_n (t)+bmt(s-m+h_n (t))\frac{\partial}{\partial t}\big(f(s-m-n,t)\big)\nonumber\\
&&\ \ \ \ \ \ \ \ \ +\,b^2mnt\frac{\partial}{\partial t}\big(f(s-m-n,t)\big)+
b^2mnt^2\frac{\partial^2}{\partial^2 t}\big(f(s-m-n,t)\big)\Big\},\\
\label{A111112}
&&L_m\cdot  W_n\cdot \big(f(s,t)\big)\nonumber\\ &=&
L_m\cdot \big(\la^nt(f(s-n,t)\big)\\
&=&\la^{m+n}\Big\{(s+h_m (t))tf(s-m-n,t)+bmtf(s-m-n,t)+bmt^2\frac{\partial}{\partial t}\big(f(s-m-n,t)\big)\Big\},\nonumber\\
\label{A111113}
&&W_n\cdot L_m\cdot \big(f(s,t)\big)\nonumber\\ &=&
W_n\cdot \big(\lambda^m\big(s+h_m (t)\big)f(s-m,t)+\lambda^mbmt\frac{\partial}{\partial t}\big(f(s-m,t)\big)\big)\nonumber\\
&=&\la^{m+n}\Big\{t(s-n+h_m (t))f(s-m-n,t)+bmt^2\frac{\partial}{\partial t}\big(f(s-m-n,t)\Big\},\\
\label{eq-alpha}
&&W_m\cdot W_n\cdot \big(f(s,t)\big)=\la^{m+n}t^2f(s-m-n,t).
\end{eqnarray}


It follows from  \eqref{A11111} that
\begin{eqnarray*}
&&[L_m, L_n]\cdot \big(f(s,t)\big)\nonumber\\
&=&
\la^{m+n}\Big\{\big((s+h_m (t))(s-m+h_n (t))-(s+h_n (t))(s-n+h_m (t))\big)f(s-m-n,t)\nonumber\\
&&\ \ \ \ \ \ \ \ \
+\,
bt\big(n(s+h_m (t)-m(s+h_n (t)\big)\frac{\partial}{\partial t}\big(f(s-m-n,t)\big)\nonumber\\
&&\ \ \ \ \ \ \ \ \
+\,
bt\big(m\big(h_n^{'} (t)\big)-n\big(h_m^{'} (t)\big)\big)f(s-m-n,t)\nonumber\\
&&\ \ \ \ \ \ \ \ \
+\,
bt\big(m(s-m+h_n (t))-n(s-n+h_m (t))\big)\frac{\partial}{\partial t}\big(f(s-m-n,t)\big)\Big\}\nonumber\\
&=&(n-m)\la^{m+n}\big(s+h_{m+n} (t)\big)f(s-m-n,t)+(n-m)(n+m)bt\frac{\partial}{\partial t}\big(f(s-m-n,t)\big)\nonumber\\
&=&(n-m)L_{m+n}\cdot\big(f(s,t)\big),
\end{eqnarray*}
where the second equality follows from
$$0=tmh_n^{'}(t)-tnh_m^{'}(t)=nh_n(t)-mh_m(t)-(n-m)h_{m+n}(t).$$

Subtracting \eqref{A111112} from \eqref{A111113} gives rise to
\begin{eqnarray*}
\!\!\!\!\!\!
[L_m,W_n]\cdot \big(f(s,t)\big)&\!\!\!=\!\!\!&
(n+bm)\la^{m+n}tf(s-m-n,t)\nonumber\\
&\!\!\!=\!\!\!&(n+bm)W_{m+n}\cdot\big(f(s,t)\big).
\end{eqnarray*}
Finally, $[W_m,W_n]\cdot \big(f(s,t)\big)=0$ follows from \eqref{eq-alpha}, completing the proof.
\end{proof}
The following proposition gives a characterization of  two $\V ir(0,b)$-modules constructed above being isomorphic.
\begin{prop}
(i) If $b=\pm1$, then
\begin{eqnarray}\label{xxit11}&\Phi(\lambda,\al,{\bf h})\cong \Phi(\mu,\be,{\bf e})\Longleftrightarrow \lambda=\mu,\al=\be,{\bf h}={\bf e},\\
\label{xxit2}&\Theta(\lambda,{\bf r})\cong\Theta(\mu,{\bf w})\Longleftrightarrow \lambda=\mu,{\bf r}={\bf w},\end{eqnarray} $\Phi(\lambda,\al,\bf h)$ is irreducible if and only if $\al\neq0$ and $\Theta(\lambda,\bf {r})$ is reducible;

(ii) If $b\neq \pm1$, then \begin{equation}\label{xxit--22}\Phi(\lambda,\al,{\bf h})\cong \Phi(\lambda,0,{\bf e}), \quad\Phi(\lambda,0,{\bf h})\cong \Phi(\mu,0,{\bf e})\Longleftrightarrow \lambda=\mu,{\bf h}={\bf e}\end{equation} and $\Phi(\lambda,0,\bf h)$ is reducible.

\end{prop}
\begin{proof} Statements concerning reducibility and irreducibility  follow directly from \cite{CG}. We only prove  \eqref{xxit11} for the case $b=1$,  similar argument can be applied to \eqref{xxit11} for the case $b=-1$,  \eqref{xxit2} and \eqref{xxit--22}. For this, it suffices to show the $``\Longrightarrow"$ part.
Let $\varphi:\Phi(\lambda,\al,{\bf h})\to\Phi(\mu,\be,{\bf e})$ be a  isomorphism of $\V ir(0,1)$-modules. Then  for any  $f(t)\in\C[t]$ and  $m\in\Z$,  we have \begin{eqnarray*}\label{xx}&\varphi(f(t))=\varphi(f(W_0)\cdot1)=f(W_0)\cdot\varphi(1)=f(t)\varphi(1),\\
& \lambda^m\varphi\big(t-(1-\delta_{m,0})\al\big)=\varphi(W_m\cdot1)=W_m\cdot \varphi(1)=\mu^m(t-(1-\delta_{m,0})\be)\varphi(1)\\
&{\rm and}\quad\label{xxp}\la^m\varphi(s+h_m(t))=\varphi(L_m\cdot1)=L_m\cdot \varphi(1)=\mu^m(s+e_m(t))\varphi(1).\end{eqnarray*}It is easy to see from the first two formulas above that   $\lambda=\mu$ and $\al=\be$.  Taking $m=1$ in the third formula one has  $\varphi(s)=s\varphi(1)$, which together with the first and third formulas gives rise to $h_m(t)=e_m(t)$ for all $m\in\Z$, i.e., ${\bf h}={\bf e}$.
\end{proof}

\section{Main result}
The main result of the present paper is to classify all modules over $U\big(\V ir(a,b)\big)$ whose restriction to $U(\C L_0 \oplus\C W_0)$ are
free of rank $1$. But we first classify such modules for $\W(a,b)$,    that is the following reslut.
\begin{theo}\label{theo1}Suppose that there exists  a $\W(a,b)$-module $M$ such that it is a free  $U(\C L_0 \oplus\C W_0)$-module of
 rank $1$. Then $a=0$, $M\cong \Phi(\lambda,\al,{\bf h})$ or $\Theta(\lambda,\bf {r})$ if $b=1$, and
 $M\cong\Phi(\lambda,\al,{\bf h})$  if $b\neq 1$ for some $\alpha\in\C, \lambda\in \C^*$ and ${\bf h}=\big(h_n(t)\big)_{n\in\Z}\in\mathcal{H}_\alpha$,\,${\bf r}=\big(r_n(t)\big)_{n\in\Z}\in \mathcal{H}_0$.
\end{theo}

\begin{proof} Let $M$ be a $\W(a,b)$-module which is a free  $U(\C L_0 \oplus\C W_0)$-module of
 rank $1$. Then $M=U(\C L_0 \oplus\C W_0)$. We divide the proof into several claims.

Formulas in the following claim can be easily shown by proceeding by induction on $i$,  which will be useful  in what follows.
\begin{clai}\label{claim1}
For any $0\le i\in\Z$ and $m\in\Z$, we have
\begin{eqnarray*}
&W_m W_0^i=W_0^iW_m, \\
&W_m L_0^i=(L_0-(a+m))^iW_m,\\
&L_m L_0^i=(L_0-m)^iL_m,\\
&L_m W_0^i=W_0^iL_m+i(a+bm)W_0^{i-1}W_m.
\end{eqnarray*}

\end{clai}
For any $m\in\Z$, assume that $g_m(L_0,W_0)=L_m\cdot1$ and  $a_m(L_0,W_0)=W_m\cdot1$ for some $g_m(L_0,W_0), a_m(L_0,W_0)\in U(\C L_0\oplus \C W_0)$. Next we are going to show that the action of $L_m$ and $W_m$ on $M$ is completely determined by $g_m(L_0,W_0)$ and  $a_m(L_0,W_0)$, respectively. Take any $u(L_0, W_0)=\sum_{i,j\geq 0}a_{i,j}L_{0}^i W_{0}^j\in U(\C L_0\oplus \C W_0)$. Then by using Claim \ref{claim1} we have
\begin{eqnarray}\label{A111112d}
&\!\!\!\!\!\!&
L_m\cdot u(L_0, W_0)\nonumber\\
&\!\!\!=\!\!\!&
L_m\cdot \sum_{i,j\geq0}a_{i,j}L_{0}^i W_{0}^j\nonumber\\
&\!\!\!=\!\!\!&\sum_{i,j\geq0}a_{i,j}(L_0-m)^iL_m\cdot W_{0}^j\nonumber\\
&\!\!\!=\!\!\!&\sum_{i,j\geq0}a_{i,j}(L_0-m)^i\big(W_{0}^jg_m(L_0,W_0) +j(a+bm)W_{0}^{j-1}a_m(L_0,W_0)\big)\nonumber\\
&\!\!\!=\!\!\!&u(L_0-m, W_0)g_m(L_0,W_0)+(a+bm)\frac{\partial}{\partial W_0}\big(u(L_0-m, W_0)\big) a_m(L_0,W_0),
\end{eqnarray}
and
\begin{eqnarray}\label{A111112c}
\!\!\!\!\!\!
W_m\cdot u(L_0, W_0)&\!\!\!=\!\!\!&
W_m\cdot \sum_{i,j\geq0}a_{i,j}L_{0}^i W_{0}^j\nonumber\\
&\!\!\!=\!\!\!&\sum_{i,j\geq0}a_{i,j}\big(L_0-(a+m)\big)^iW_m\cdot W_{0}^j\nonumber\\
&\!\!\!=\!\!\!&u(L_0-a-m, W_0)a_m(L_0,W_0).
\end{eqnarray}
Hence in order to finish the proof it suffices to determine $a_m(L_0,W_0)$ and $g_m(L_0,W_0)$ for all $m\in\Z$.
\begin{clai}\label{claim3}For all $m\in\Z$, $a_m(W_0):=a_m(L_0, W_0)\in \C[W_0]$. Moreover, if there exists $a_m(L_0, W_0)=0$ for some $m\in\Z$, then  $(a,b)=(0,1)$ and in this case $a_n(W_0)=\delta_{n,0}W_0$  for all $n\in\Z$.
\end{clai}
For the case $(a,b)=(0,1)$, if there exists some $m_{0}\in\Z^*$ such that $a_{m_0}(L_0,W_0)=0$, then using $[L_n,W_{m_0}]=(n+m_0)W_{n+m_0} $ one has
$a_m(L_0,W_0)=0$ for all $m\in\Z^*$.

It remains to show that $a_m(L_0, W_0)\neq 0$ for any $m\in\Z$ in the case $(a,b)\neq (0,1)$.
 Suppose on the contrary that $a_m(L_0,W_0)=0$ for some $m\in\Z^*$.  Set $I_m=\{n\in\Z\mid a+bn+m\neq 0\}$. Note that $I_m$ is an infinite set. It follows from  Claim \ref{claim1} that  $W_m\cdot M=0$,  which together with \eqref{rel2} gives rise to
$W_{m+n}\cdot M=0$ for all $n\in I_m$. Then by \eqref{rel2} again, we have $$0=[L_{-m-n}, W_{m+n}]=(a-b(m+n)+m+n)W_0.$$ Now this and the infinity of $I_{m}$ imply $a_0(L_0,W_0)=0$, which is absurd.  Now we can assume $a_m(L_0,W_0)=\sum_{i=0}^{k_m}b_{m,i}L_0^i$ for some $b_{m,i}=b_{m,i}(W_0)\in\C[W_0]$ and $b_{m,k_m}\neq 0$.
Note that \begin{eqnarray*}
\!\!\!\!\!\!
0&\!\!\!=\!\!\!&
W_m\cdot W_n\cdot1-W_n\cdot W_m\cdot1\nonumber\\
&\!\!\!=\!\!\!&\sum_{i=0}^{k_n}b_{n,i}(L_0-(a+m))^i\sum_{i=0}^{k_m}b_{m,i}L_0^i-\sum_{i=0}^{k_m}b_{m,i}(L_0-(a+n))^i\sum_{i=0}^{k_n}b_{n,i}L_0^i\nonumber\\
&\!\!\!\equiv\!\!\!&b_{m,k_m}b_{n,k_n}((a+n)k_m-(a+m)k_n)L_0^{k_n+k_m-1}\quad \big({\rm mod}\oplus_{i=0}^{k_n+k_m-2}\C[W_0] L_0^i\big).
\end{eqnarray*}
Hence $k_m=0$ for any $m\in\Z$, i.e., $a_m(L_0,W_0)\in \C[W_0]$.

\begin{clai}
$a=0.$
\end{clai}

Combing the relation \eqref{rel2} with \eqref{A111112d}-\eqref{A111112c} and Claim \ref{claim3}, we obtain
\begin{eqnarray}\label{A111112de}
\!\!\!\!\!\!
(a+n+bm)a_{m+n}(W_0)&\!\!\!=\!\!\!&
a_n(W_0)\big(g_m(L_0,W_0)-g_m(L_0-a-n,W_0)\big)\nonumber\\
&\!\!\!\!\!\!\!\!\!\!\!\!\!\!\!\!\!\!\!\!\!\!\!\!&
+
(a+bm)a_n^\prime(W_0) a_m(W_0).
\end{eqnarray}
Thus we must have  ${\rm deg}_{L_0}(g_m(L_0,W_0)-g_m(L_0-a-n, W_0))=0$, which in turn requires ${\rm deg}_{L_0}g_m(L_0,W_0)\leq1$ for all $m\in\Z$. These entail us to assume that \begin{equation}\label{eqq}g_m(L_0,W_0)=b_mL_0+d_m,\end{equation} for any $m\in\Z$ and  some $b_m, d_m\in \C[W_0]$. Substituting \eqref{eqq} into \eqref{A111112de} one has
\begin{eqnarray}\label{eq3}
&&(a+n+bm)a_{m+n}(W_0)\nonumber\\
&=&(a+n)a_n(W_0)b_m+(a+bm)a_n^\prime(W_0) a_m(W_0).
\end{eqnarray}Setting $m=0$ in the equation above one  has $$aa_n^\prime(W_0)W_0=0,$$ by noticing  $a_0(W_0)=W_0$ and $b_0=1$.    Thus $a=0$.

\begin{clai}\label{clai22}
For all $m\in\Z$,  $b_m=\lambda^m$ for some $\lambda\in\C^*$.
\end{clai}
From $[L_m, L_n]\cdot1=(n-m)L_{m+n}\cdot1$, we have for all $m,n\in\Z$ that

\begin{eqnarray}\label{eq-l_ml_n}
\!\!\!\!\!\!
(n-m)\big(b_{m+n}L_0+d_{m+n}\big)&\!\!\!=\!\!\!&
(n-m)b_mb_nL_0+\Big(nb_md_n-mb_nd_m\nonumber\\
&\!\!\!\!\!\!\!\!\!\!\!\!\!\!\!\!\!\!\!\!\!\!\!\!&
-
nbd_m^\prime a_n(W_0)+mbd_n^\prime a_m(W_0)\Big).
\end{eqnarray}
By comparing the coefficients of $L_0$ of both sides of \eqref{eq-l_ml_n} and setting $m=1$ and $m=-1$, we respectively get
\begin{eqnarray*}
(n-1)b_{n+1}=(n-1)b_nb_1\quad {\rm and}\quad (n+1)b_{n-1}=(n+1)b_nb_{-1},
\end{eqnarray*} from which one has $b_n=\lambda^n$ $(\lambda:=b_1\in\C^*)$ for all $n\in\Z$. Thus the claim holds.

\begin{clai}\label{clai6}If $a_m(W_0)\neq 0$, then
${\rm deg}_{W_0}a_m(W_0)=1$.
 \end{clai}An important fact observed from \eqref{eq3} is that the degrees of $a_n(W_0)$ are bounded.  Assume that ${\rm deg}_{W_0}a_{m_0}(W_0)={\rm max}\{{\rm deg}_{W_0}a_{k}(W_0)\mid k\in\Z\}.$  Then we assert that ${\rm deg}_{W_0}a_{m_0}(W_0)\leq1$. If $m_0=0$, then we are done. For $m_0\neq0$, setting $n=m=m_0$ in \eqref{eq3} along with Claim \ref{clai22} gives rise to \begin{equation}\label{expofb_m}
\la^{m_0}=(1+b)\frac{a_{2m_0}(W_0)}{a_{m_0}(W_0)}-ba_{m_0}^{\prime}(W_0).
\end{equation}Then we must have $\frac{a_{2m_0}(W_0)}{a_{m_0}(W_0)}\in\C^*$ by the choice of $m_0$, which implies ${\rm deg}_{W_0}a_{m_0}(W_0)\leq1$ for $b\neq0$. While in the case  $b=0$,  \eqref{eq3} simply becomes $$na_{m+n}(W_0)=na_n(W_0)\la^m,$$ from which we obtain $a_n(W_0)=\la^nW_0$ for all $n\in\Z$. So in either case our assertion is true.  This allows us to assume that  \begin{equation}\label{eq-1}a_m(W_0)=A_mW_0+B_m, \end{equation}for some $A_m,B_m\in\C$. Substituting \eqref{eq-1} into \eqref{eq3} implies that
\begin{eqnarray}
\label{eqn1}(n+bm)A_{m+n}\!\!\!&=&\!\!\!nA_n\la^m+bmA_nA_m, \\
\label{eqn2} (n+bm)B_{m+n}\!\!\!&=&\!\!\!nB_n\la^m+bmA_nB_m, \quad \forall m,n\in\Z.
\end{eqnarray}

We next assert that $A_n\in\C^*$ for all $n\in\Z$. Suppose on the contrary that there exists some $n_0\in\Z^*$ such that $A_{n_0}=0$.  Letting $n=-m=n_0$ in \eqref{eqn1} along with the fact that $A_0=1$, we have $b=1$. Then it follows from   by setting $n=n_0$ in \eqref{eqn1} that $A_n=0$ for all $n\in\Z^*$. Setting $n+m=0$ in \eqref{eqn2}, then $A_n=0$ implies $B_n=0$ for all $n\in\Z^*$. Hence,  $a_n(W_0)=0$ for all $n\in\Z^*$,  contradicting the assumption of our claim. Now Claim \ref{clai6} follows immediately from the two assertions.

\begin{clai}\label{clai223}
For all $m\in\Z$,  $A_m=\lambda^m$, where $\lambda\in \C^*$ as in Claim \ref{clai22}.
\end{clai}
Letting $n=-m=-1$ and $n=-m=1$ in \eqref{eqn1} respectively, we obtain
\begin{eqnarray}\label{ass}
b-1=A_{-1}(b A_1-\lambda)\quad {\rm and}\quad b-1=A_1(b A_{-1}-\lambda^{-1}).
\end{eqnarray}
Equating the above two equations gives rise to $A_1\lambda^{-1}=A_{-1}\lambda$, i.e.,
\begin{equation}\label{eq=+-}
A_1=k \lambda,\, A_{-1}=k\lambda^{-1} \mbox { \ for some \ } k\in\C^*.
\end{equation}
Substituting \eqref{eq=+-} into one of the equations in \eqref{ass} implies $bk^2-k-(b-1)=0$, from which we obtain
\begin{equation*}
k=\left\{\begin{array}{llll}1&\mbox{if \ }b\in\{0, \frac{1}{2}\},\\[4pt]
1\, {\rm or\,} \,\frac{1-b}{b}&\mbox{otherwise\ }.
\end{array}\right.\end{equation*}
Indeed, $k\equiv1$, i.e., the case $k=\frac{1-b}{b}$ for some $b\not\in \{0, \frac{1}{2}\}$ can not occur. Suppose not, setting $n=m=1$ in \eqref{eqn1} along with \eqref{eq=+-} forces that
\begin{equation}\label{eq=+-1}b(1+b)A_2=(b-1)(b-2)\lambda^2.\end{equation}
But then $b=-1$ would give $\lambda=0$, contradicting  the choice of $\lambda$; while $b\neq -1$ would give rise to $b=\frac12$ by taking $n=-1, m=2$ in
\eqref{eqn1} and using \eqref{eq=+-1}, a contradiction with  $b\notin\{0,\frac12\}$.

Now \eqref{eq=+-} turns out to be
\begin{equation}\label{A_1A_{-1}}  A_1=\lambda, A_{-1}=\lambda^{-1}
\end{equation} due to $k=1$. Taking $n=1$ and $n=-1$ in \eqref{eqn1} respectively, we have
\begin{eqnarray}
\label{eqn11az}&(bm+1)A_{m+1}=\lambda^{m+1}+bm\lambda A_{m}, \\
\label{eqn22az}&(bm-1)A_{m-1}=-\lambda^{m-1}+bm\lambda^{-1}A_{m}.
\end{eqnarray}
If $1+bm\neq0$ for all $m\in\Z$, then \eqref{eqn11az} together with \eqref{A_1A_{-1}} gives  $A_m=\lambda^m$ for any $m\in\Z$. Otherwise, assume $b=-\frac{1}{m_0}$ for some $m_0\in\Z^*$. Consider the case $m_0>0$. From \eqref{eqn11az} we have $A_{m}=\lambda^m$ for all $m\in\Z$ and $m<1+m_0$. Letting $m=m_0+1$ in \eqref{eqn22az}, we have $(b-2)A_{m_0}=-\lambda^{m_0}+(b-1)\lambda^{-1}A_{1+m_0}$,  which forces $A_{1+m_0}=\lambda^{m_0+1}$. By \eqref{eqn11az} again we obtain
$A_m=\lambda^m$ for all $m\in\Z$. Similarly one can show that the claim is also true for the case $m_0<0$.

\begin{clai}\label{clai7}If $a_m(W_0)\neq 0$ for all $m\in\Z$, then
$a_m(W_0)=\lambda^m\big(W_0-\delta_{b,-1}m\al-\delta_{b,1}(1-\delta_{m,0})\al\big)$ for some  $\al\in\C$.
\end{clai}
Combining Claim \ref{clai22} with \eqref{eqn2} one has
\begin{equation}\label{eq=+--1}(n+bm)B_{m+n}=n\lambda^m B_n+bm\lambda^n B_m\quad {\rm for\ all}\ m,n\in\Z.\end{equation} Setting $n=-m=-1$ and $n=-m=1$ in \eqref{eq=+--1} respectively, we obtain
$$B_{-1}\lambda=b\lambda^{-1}B_{1}\quad {\rm and}\quad b B_{-1}\lambda=\lambda^{-1}B_{1},$$
from which we get $B_1\lambda^{-1}(b^2-1)=0$. There are three cases to be considered.
\begin{case} $b\neq \pm1$. \end{case}

In this case, we have $B_1=B_{-1}=0$. In fact,  $B_{m}=0$ for all $m\in \Z$. To see this,  taking $n=1$ and $n=-1$ in \eqref{eq=+--1} respectively, we have
\begin{eqnarray}
\label{eqn11a}(bm+1)B_{m+1}\!\!\! &=&\!\!\! bm\lambda B_{m}, \\
\label{eqn22a}(bm-1)B_{m-1}\!\!\! &=&\!\!\! bm\lambda^{-1}B_{m}, \quad \forall m\in\Z.
\end{eqnarray}
 Then either relation above would give $B_m=0$ for all $m\in\Z$ if  $1+bm\neq0$ for all $m\in\Z$. Without loss of generality, we only  consider $b=-\frac{1}{m_0}$ for some  positive integer $m_0$.
It follows from  \eqref{eqn11a} and  $B_1=0$ that one can inductively show $B_{m}=0$ for all $1\le m<1+m_0$. While $B_{1+m_0}=0$ can be obtained immediately from  by taking $m=m_0+1$ in \eqref{eqn22a}. But this and \eqref{eqn11a} give rise to $B_{m}=0$ for all $m\ge 1+m_0$. Hence, $B_m=0$ for all $m\in\Z$, since $B_m=0$ for all negative integers can be inferred from \eqref{eqn22a}.
\begin{case} $b=1$. \end{case}
In this case, from \eqref{eq=+--1} one can show by induction that $B_m=(1-\delta_{m,0})B_1 \la^{m-1}$ for $m\in\Z$.
\begin{case} $b=-1$. \end{case}
Inductively,  we have $B_m=mB_1 \la^{m-1}$ for all $m\in\Z$  by \eqref{eq=+--1}.

\bigskip
The three cases above can be summarized as follows:
\begin{equation*}
a_m(W_0)=\left\{\begin{array}{llll}\lambda^m W_0+(1-\delta_{m,0})B_1 \la^{m-1}&\mbox{if \ }b=1,\\[4pt]
\lambda^m W_0+mB_1 \la^{m-1}&\mbox{if \ }b=-1,\\[4pt]
\lambda^m W_0&\mbox{otherwise\ }.
\end{array}\right.\end{equation*}
Whence Claim \ref{clai7} follows by setting $\al=-\frac{B_1}{\lambda}$.

It now follows from \eqref{eqq} and Claim \ref{clai22} that $g_n(L_0, W_0)=\lambda^n L_0+d_n$ for all $n\in\Z$. So it remains to determine the expressions of these $d_n$.
\begin{clai}\label{clai8}
\begin{equation*}
g_n(L_0, W_0)=\left\{\begin{array}{llll}\la^n\big(L_0+r_n(W_0)\big)&\mbox{\rm if \ }(a,b)=(0,1)\ {\rm and}\ a_n(W_0)=\delta_{n,0}W_0,\\[4pt]
\la^n\big(L_0+h_n(W_0)\big)&\mbox{\rm if \ } (a,b)=(0,1)\ {\rm and}\ a_n(W_0)\neq0,\\[4pt]
\la^n\big(L_0+h_n(W_0)\big) &\mbox{\rm if \ }(a,b)\neq(0,1)
\end{array}\right.\end{equation*}
for some  $\big(h_n(t)\big)_{n\in\Z}\in\mathcal{H}_{\al}$ and $\big(r_n(t)\big)_{n\in\Z}\in\mathcal H_0$ {\rm (cf. \eqref{noas})}.
\end{clai}

It follows from Claim \ref{clai22} and  by comparing monomials concerning $W_0$ in \eqref{eq-l_ml_n} that
\begin{equation}\label{eq=+--1f}n\la^m d_n-m\la^n d_m-nbd_m^\prime a_n(W_0)+mbd_n^\prime a_m(W_0)=(n-m)d_{m+n}.\end{equation}
Denote $F_m=\la^{-m}d_{m}\in\C[W_0].$

Now  we assume that $a_m(W_0)\neq 0$ for all $m\in\Z$.    Then  by Claim \ref{clai7},  \eqref{eq=+--1f} can be rewritten as
\begin{eqnarray}\label{ggg}
\!\!\!\!\!\!
(n-m)F_{m+n}&\!\!\!=\!\!\!&
nF_n-mF_m-nbF_m^\prime\big(W_0-\delta_{b,-1}n\al-\delta_{b,1}(1-\delta_{n,0})\al\big)\nonumber\\
&\!\!\!\!\!\!\!\!\!\!\!\!\!\!\!\!\!\!\!\!\!\!\!\!&
+
mbF_n^\prime\big(W_0-\delta_{b,-1}m\al-\delta_{b,1}(1-\delta_{m,0})\al\big).
\end{eqnarray}
 It is important to observe  that  degrees of all $F_n$ are bounded. Let $k$ be a nonnegative integer such that ${\rm deg\,}F_n\leq k$ for all $n\in\Z$.
Let $f_n$ be the coefficient of $W_0^k$ of $F_n$. By Claim 7 of \cite{CG}, we see that $f_n=nf_1$ for all $n\in\Z$. It follows from Proposition \ref{Pro}  that all $q_{n,k;\al}$ in \eqref{ew} satisfy \eqref{ggg}.
Then by the construction of $q_{n,k;\al}$ there exists ${h}^{(k)}\in\C$ such that ${\rm deg}_{W_0}(F_n-{h}^{(k)}q_{n,k;\al})\leq k-1$ for all $n\in\Z$. Observe that $F_n-{h}^{(k)}q_{n,k;\al}$ for all $n\in\Z$ also satisfy \eqref{ggg}. Replacing $F_n$ with $F_n-{h}^{(k)}q_{n,k;\al}$ and repeating the above process, we can find ${h}^{(1)}, {h}^{(2)},...,{h}^{(k-1)}\in\C$ such that $F_n-\sum_{i=1}^k{h}^{(i)}q_{n,i;\al}$ are constants and satisfy \eqref{ggg}, that is,
$$F_n-\sum_{i=1}^k{h}^{(i)}q_{n,i;\al}={h}^{(0)} q_{n,0;\al}$$
for some ${h}^{(0)}\in\C$. Hence,  $F_n=\sum_{i=0}^k{h}^{(i)}q_{n,i;\al}$.

Similarly, in  the case  $a_m(W_0)=\delta_{m,0}W_0$ for any $m\in\Z$ one can show that there exists $\big(r_n(t)\big)_{n\in\Z}\in\mathcal H_0$ such that $$g_n(L_0, W_0)=\la^n\big(L_0+r_n(W_0)\big)$$ for all $n\in\Z$.
This completes the proof of Claim \ref{clai8} and then Theorem \ref{theo1}.\end{proof}

When $\C[s,t]$ is considered as a $\V ir(a,b)$-module, i.e., taking the central elements $C_i$ into account, it can be showed by the similar arguments as above, but whose proof is a little more complicated than that of Theorem \ref{theo1}, that $C_i=0$ on $\C[s,t]$. So   the following result also holds.


\begin{theo}Suppose that there exists  a $\V ir(a,b)$-module $M$ such that it is a free  $U(\C L_0 \oplus\C W_0)$-module of
 rank $1$. Then $a=0$, $M\cong \Phi(\lambda,\al,{\bf h})$ or $\Theta(\lambda,\bf {r})$ if $b=1$, and
 $M\cong\Phi(\lambda,\al,{\bf h})$  if $b\neq 1$ for some $\alpha\in\C, \lambda\in \C^*$ and ${\bf h}=\big(h_n(t)\big)_{n\in\Z}\in\mathcal{H}_\alpha$,\,${\bf r}=\big(r_n(t)\big)_{n\in\Z}\in \mathcal{H}_0$.
\end{theo}



\end{CJK*}
\end{document}